\documentclass[12pt]{article}
\usepackage{amsmath}
\usepackage{amsfonts}
\usepackage{cite}
\usepackage{graphicx}
\usepackage[margin=25mm]{geometry}
\usepackage{amsmath,amssymb,amsfonts,amsthm}
\usepackage{graphicx}
\usepackage{multirow}
\usepackage{bm}
\newtheorem{lemma}{Lemma}[section]
\newtheorem{theorem}{Theorem}[section]
\newtheorem{corollary}{Corollary}[section]
\def\bmath#1{\mbox{\boldmath $#1$}}
\newcommand{\eset}[1]{\left\lbrace #1 \right\rbrace}

\DeclareMathOperator*{\argmin}{arg\,min}

\newcommand{\getheight}{\rule{0em}{2.5ex}}

\def\0{\bmath{0}}
\def\1{\bmath{1}}

\title{Bounds for the number of basic feasible solutions generated by the simplex method with the largest distance rule}
\author{Tomonari Kitahara\footnote{Faculty of Economics, Kyushu University, Fukuoka 819-0395, Japan. E-mail: tomonari.kitahara@econ.kyushu-u.ac.jp}}
\date{\today}

\begin{document}

\maketitle
\abstract{
In this paper, we analyze the simplex method with the largest distance rule and derive upper bounds on the number of different basic feasible solutions generated. The pivoting rule was proposed by Pan \cite{pan2008}, and in some cases, it was reported to be more efficient than the renowned steepest edge rule. We show that the analytical framework developed by Kitahara and Mizuno can be extended to this rule, despite its structural differences from previously studied pivoting rules. The resulting bounds involve a geometric parameter $\beta$ determined by the column norms of the constraint matrix. In addition, our analysis does not require a nondegeneracy assumption.
}
\section{Introduction}
\subsection{Background}
The simplex method is one of the most efficient algorithms for solving linear programming problems (LPs). 
Its efficiency is greatly affected by a pivoting rule. 
There are some important pivoting rules from a practical and theoretical viewpoint. 
Dantzig's rule (or the most negative coefficient rule) \cite{dantzig1963} is the oldest and the best-known pivoting rule. 
When the problem in consideration is degenerate, the simplex method might encounter a phenomenon known as cycling, in which the algorithm visits the same set of bases infinitely.
Bland's rule \cite{bland1977} is the first countermeasure for cycling. 
Similarly, the shadow pivoting rule is also important from a theoretical standpoint.
Borgwardt \cite{borgwardt1987} showed that for random LPs, the expected value of the number of iterations is strongly polynomial in 
the number of facets of the feasible region and the number of variables. 
Spielman and Teng \cite{spielman2004} introduced smoothed analysis and showed that the simplex method runs in polynomial time under random perturbations.
From a practical view point, it has been widely observed that the steepest edge rule shows superior performance to other pivoting rules \cite{matousek2007}.  
Pan \cite{pan2008} proposed a pivoting rule called the largest distance rule.
Geometrically speaking, this pivoting rule chooses, as an entering variable, a nonbasic variable whose corresponding dual constraint has the largest unit violation. 
Through his computational experiments, Pan \cite{pan2010} showed that a variant of the largest distance rule shows better performance than the steepest edge rule.\\

To assess deterministic complexity of the simplex method, Kitahara and Mizuno ~\cite{km2013} analyzed 
the number of basic feasible solutions (BFSs) generated by the simplex method.
They showed upper bounds for the number of different BFSs generated by the simplex method with Dantzig's rule, namely
\begin{equation} \notag
  \left\lceil \frac{m\gamma}{\delta} \log\left( \frac{\getheight c^{\top}x^0-z^*}{\getheight c^{\top}\bar{x}-z^*} \right) \right\rceil \quad \mathrm{and}~
 (n-m) \left\lceil \frac{m\gamma}{\delta} \log\left(\frac{m\gamma}{\delta}\right) \right\rceil,
\end{equation}
where $m$ is the number of constraints, $n$ is the number of variables, $c$~is the coefficient vector of the objective function, $x^0$~is an initial BFS, $\bar{x}$~is a BFS with the second smallest objective value, $z^*$~is the optimal value, and $\gamma$ and~$\delta$ are the maximum and minimum positive elements, respectively, in all BFSs
\footnote{In the original paper, the latter bound is $n \left\lceil \frac{m\gamma}{\delta} \log\left(\frac{m\gamma}{\delta}\right) \right\rceil$. 
But this can be slightly improved to $ (n-m) \left\lceil \frac{m\gamma}{\delta} \log\left(\frac{m\gamma}{\delta}\right) \right\rceil$ by a trivial observation. 
See the proof of Theorem  \ref{thm:Upper-Bound-Without-Objective-Function-Ratio}.}.
Assuming nondegenerate LPs, these bounds become upper bounds for the number of iterations of the simplex methods with Dantzig's rule.
Inspired by these results, Tano et al. \cite{tano2019} showed bounds for the number of iterations of the simplex method with the steepest edge rule 
when it is applied to a nondegenerate LP.
The bounds are 
\[
\left\lceil m^{\frac{3}{2}} \frac{\gamma^2}{\delta^2} \log\left( \frac{\getheight c^{\top}x^0-z^*}{\getheight c^{\top}\bar{x}-z^*} \right) \right\rceil~\mathrm{and}~
(n-m) \left\lceil m^{\frac{3}{2}} \frac{\gamma^2}{\delta^2} \log\left(\frac{m\gamma}{\delta}\right) \right\rceil.
\]  
However, it is not clear to what extent this analytical framework can be applied to other pivoting rules, especially those whose selection criteria depend on the original constraint matrix rather than the transformed nonbasic matrix.
\subsection{Main contributions}
In this paper, we show that the analytical framework developed by Kitahara and Mizuno can be extended to the largest distance rule.
Although this rule has a simple form, its structure differs from previously analyzed pivoting rules, since it depends on the norm of the original constraint matrix rather than the transformed nonbasic matrix. We show that this difference leads to complexity bounds involving a geometric parameter $\beta$, defined by the ratio of column norms of the constraint matrix.
As a result, we derive upper bounds for the number of different basic feasible solutions generated by the simplex method with the largest distance rule. In addition, unlike previous analyses for the steepest edge rule, our results do not require a nondegeneracy assumption.
\subsection{Notations and the structure of the paper}
We use $0$ both for a number and a zero vector. 
We believe the distinction is clear from context.
For a vector $v$, $v_i$ means the $i$ th element of $v$. 
For a vector or a matrix $v$, $v^{\top}$ represents the transpose of $v$. 
For vectors $v,w\in \mathbb{R}^n$, $v=w$ means $v_i=w_i$ for all $i\in\{1,2\dots,n\}$. 
Also, $v\ge w$ means $v_i\ge w_i$ for all $i\in\{1,2\dots,n\}$. 
For a real number $r$, $\lceil r\rceil$ is the minimum integer greater than $r$.\\ 

This paper is organized as follows. 
In Section 2, we give preliminaries for the analysis in this paper. 
We also explain existing research related to the topic of this paper.  
In Section 3, we analyze the simplex method with the largest distance rule and obtain the above two bounds for the number of different BFSs. 
Finally in Section 4, we conclude the paper.
\section{
%Preliminaries
Preliminaries and existing research
}
\subsection{Linear programming and the simplex method} \label{sec:LP}
Let $m<n$ be positive integers. We consider an LP with $n$~nonnegative variables and $m$~constraints:
\begin{equation} \label{primal}
\begin{array}{ll}
\min&c^{\top}x,\\
\text{subject to}&Ax=b,\\
&x\ge0,
\end{array}
\end{equation}
where $A \in \mathbb{R}^{m \times n}$, $b \in \mathbb{R}^m$, and $c \in \mathbb{R}^n$ are given data and $x \in \mathbb{R}^n$ is a variable vector.
In this paper, we assume $\mathrm{rank}(A)=m$.  
We will summarize assumptions of this paper in Section \ref{assumptions}. 
The dual of problem~\eqref{primal} is expressed as
\begin{equation} \label{dual}
\begin{array}{ll}
\max&b^{\top}y,\\
\text{subject to}&A^{\top}y+s=c,\\
&s\ge0,
\end{array}
\end{equation}
where $y \in \mathbb{R}^m$ and $s \in \mathbb{R}^n$ are variable vectors.\\
%Problem~\eqref{eq:LP-Primal} is called the primal problem with regard to the dual problem~\eqref{eq:LP-Dual}.

Next, we define a dictionary. 
Let $B\subset\{1,2\dots,n\}$ and $N=\{1,2,\dots,n\}\setminus B$.
According to $B$ and $N$, we split $A$, $c$, and~$x$ as follows:
\[
A= \begin{bmatrix}
    A_B & A_N
  \end{bmatrix},~
  c = \begin{bmatrix}
    c_B\\
    c_N
  \end{bmatrix},~
  x = \begin{bmatrix}
    x_B\\
    x_N
  \end{bmatrix}.
\]
By using these splits, problem~\eqref{primal} is written as
\begin{equation} \label{primal_r}
\begin{array}{ll}
\min&c_B^{\top}x_B+c_N^{\top}x_N,\\
\text{subject to}&A_Bx_B+A_Nx_N=b,\\
&x=(x_B,x_N)\ge0.
\end{array}
\end{equation}

When $|B| = m$ and $A_B$ is nonsingular, $B$ and $N$ are called a basis and a nonbasis, respectively. 
In this case, problem~\eqref{primal_r} is transformed by multiplying the equality constraint by~$A_B^{-1}$ from the left
and rearranging appropriately:
\begin{equation} \label{primal_dictionary}
\begin{array}{ll}
\min& c_B^{\top}A_B^{-1}b+\bar{c}_N^{\top}x_N,\\
\text{subject~to}&x_B=A_B^{-1}b-\bar{A}_Nx_N,\\
&x=(x_B,x_N)\ge0,
\end{array}
\end{equation}
where $\bar{c}_N = c_N - A_N^{\top} (A_B^{\top})^{-1}c_B$ and $\bar{A}_N = A_B^{-1} A_N$.
Formulation~\eqref{primal_dictionary} is called a dictionary for a basis~$B$.
The vector $\bar{c}_N$ is called a reduced cost vector and $\bar{A}_N$ is called a nonbasic matrix.
A basic solution is a solution $x$ such that $(x_B,x_N) = (A_B^{-1}b, 0)$ in dictionary (\ref{primal_dictionary}).
The elements of $x_B$ are basic variables and those of $x_N$ are nonbasic variables.
A basic solution~$x$ is called a basic feasible solution (BFS) if $x_B\ge0$.
A basis~$B$ and a nonbasis~$N$ are called a feasible basis and a feasible nonbasis, respectively, if the corresponding basic solution is feasible.\\

For a given BFS in dictionary~\eqref{primal_dictionary}, if $\bar{c}_N \ge 0$, then the current BFS is optimal. 
Let us assume otherwise. 
In this case, we first choose a nonbasic variable $x_j,~j\in N,$ whose reduced cost is negative. 
By increasing the value of $x_j$ from 0, we can reduce the objective function. 
If we can increase the value of $x_j$ infinitely without violating the constraints, then we declare that the problem (\ref{primal}) is unbounded and stop the algorithm. 
Otherwise, by increasing the value of $x_j$, the value of some basic variable $x_i$ becomes zero at some point. 
Then we exchange the roles of $x_i$ and $x_j$ to obtain a new dictionary.
The overall procedure is called a pivot.
In this example, $x_j$ is called an entering variable and $x_i$ is called a leaving variable.

\subsection{Pivoting rules} \label{Sec:pivoting-rules}

A rule for choosing an entering variable from nonbasic variables with negative reduced costs is called a pivoting rule.
Many pivoting rules have been proposed and they greatly affect the efficiency of the simplex method.
In this subsection, we describe Dantzig's rule (or the most negative coefficient rule), the steepest-edge rule and the largest distance rule.\\

{\bf Dantzig's rule \cite{dantzig1963}}: Dantzig's rule chooses a nonbasic variable with the smallest reduced cost as an entering variable.
For given feasible dictionary~\eqref{primal_dictionary}, this rule chooses an entering variable $x_{j_d}$ such that
\[
j_d \in \argmin\{\bar{c}_j|~j\in N\}
\]
\\

{\bf Steepest edge rule \cite{goldfarb1977, forrest1992}}: The steepest-edge rule focuses on the difference vector of basic solutions and the decrease in the objective value. 
Formally, it chooses a nonbasic variable $x_{j_s}$ which satisfies
\[
j_s\in\argmin\{\frac{\bar{c}_j}{\sqrt{1+\|\bar{a}_j\|^2}}|~j\in N\},
\]
where $\bar{a}_j$ represents $j$-th column vector of the nonbasic matrix $\bar{A}_N$. 
It is easy to see the fraction in the bracket is the reduction of the objective function divided by the norm of the displacement vector, when we choose $x_j$ as an entering variable. 
\\

{\bf largest distance rule \cite{pan2008}}: Under the largest distance rule, we choose a nonbasic variable $x_{j_m}$ satisfying
\begin{equation}
j_m \in \argmin\{\frac{\bar{c}_j}{\|a_j\|}|~j\in N\}, \label{largest}
\end{equation}
where $a_j$ represents the $j$-th column vector of the constraint matrix $A$. 
Geometrically speaking, this rule commands to choose a nonbasic variable whose corresponding dual constraint has the largest unit violation. 
Note that the steepest edge rule and the largest distance rule take a similar form. 
While $\bar{a}_j$ in the steepest rule varies iteration to iteration and computation of $\bar{a}_j$ is costly, $a_j$ in the largest distance rule is constant. 
Pan \cite{pan2010} reported a variant of the largest distance rule shows superior performance over the steepest edge rule in his computational experiments.
\subsection{Parameters}\label{parameters}
In this subsection, we explain parameters needed for our analysis. \\

First, let $\delta$ and $\gamma$ be  the minimum and maximum positive elements, respectively, in all BFSs. 
From the definition, for a BFS $(x_B,x_N)$ and $j \in B$ such that $x_j\not=0$, we have
\[
\delta\le x_j\le \gamma.
\]
We also need another parameter $\beta$, which is defined as 
\begin{equation}
\beta=\frac{\min_{j\in \{1,2,\dots_n\}}\|a_j\|}{\max_{j\in \{1,2,\dots_n\}}\|a_j\|}, \label{beta}
\end{equation}
where $a_j$ denote the $j$ th column vector of $A$.  
Note that $\beta$ can be directly computed from $A$. 
Note also that $\beta$ can be bounded from below as 
\[
\beta\ge \frac{\min_{i,j\in \{1,2,\dots_n\}}|a_{ij}|}{\sqrt{n}\max_{i,j\in \{1,2,\dots_n\}}|a_{ij}|},
\]
where $a_{ij}$ means the $(i,j)$ element of $A$.
\subsection{Assumptions}\label{assumptions}
Throughout this paper, we pose the following three assumptions.
\begin{itemize}
  \item[(i)] $\text{rank~} A = m$;
  \item[(ii)] problems~\eqref{primal} and~\eqref{dual} have optimal basic solutions denoted by $x^*$ and $(y^*, s^*)$, respectively, and these optimal values are~$z^*$;
  \item[(iii)] an initial BFS~$x^0$ is available and it is not optimal, that is, the objective value of~$x^{0}$ is larger than~$z^*$.
\end{itemize}
Concerning (ii), we use $\bar{x}$ to represent a BFS with the second smallest objective value.
These assumptions are the same as those made in the previous study~\cite{km2013}; 
%Tano et al. make an additional assumption that the problem is nondegenerate. 
\subsection{Previous research}
%Although various pivoting rules have been proposed, it is still an open problem whether there exists a pivoting rule with which the simplex method runs in polynomial time with regard to $m$ and~$n$.
To consider the complexity of the simplex method, some recent studies analyzed the number of different BFSs generated by the simplex method.
Kitahara and Mizuno~\cite{km2013} proved that an upper bound for the number of different BFSs generated by the simplex method with Dantzig's rule is
\begin{equation} \label{eq:bound-by-Kitahara-1}
  \left\lceil \frac{m\gamma}{\delta} \log\left( \frac{c^{\top} x^0 - z^*}{c^{\top} \bar{x} - z^*} \right) \right\rceil
\end{equation}
for a standard LP with $n$~variables and $m$~constraints (see Section~\ref{parameters} for the notation used in \eqref{eq:bound-by-Kitahara-1} and~\eqref{eq:bound-by-Kitahara-2}.)
Another upper bound that is independent of the objective value was also shown in~\cite{km2013}:
\begin{equation} \label{eq:bound-by-Kitahara-2}
  (n-m) \left\lceil \frac{m\gamma}{\delta} \log\left( \frac{m\gamma}{\delta} \right) \right\rceil.
\end{equation}
For a nondegenerate LP, upper bounds \eqref{eq:bound-by-Kitahara-1} and~\eqref{eq:bound-by-Kitahara-2} can be regarded as upper bounds for the number of iterations of the simplex method with Dantzig's rule.
The same authors \cite{km2011} observed that for a variant of Klee-Minty problem, the bound \eqref{eq:bound-by-Kitahara-2} is close to the actual number of iterations.\\
%Additionally, Kitahara and Mizuno showed that for a nondegenerate LP, upper bounds \eqref{eq:bound-by-Kitahara-1} 
%and~\eqref{eq:bound-by-Kitahara-2} are also valid for the simplex method with the best improvement rule.
%We can show upper bounds \eqref{eq:bound-by-Kitahara-1} and \eqref{eq:bound-by-Kitahara-2} are also valid for the simplex method with the best improvement rule.

Kitahara and Mizuno~\cite{km2013a} also studied pivoting rules that do not increase the objective value in each iteration.
Using such a pivoting rule, the simplex method generates at most
\begin{equation} \label{eq:bound-by-Kitahara-3}
  \left\lceil \min\eset{m,n-m} \frac{\gamma \gamma^{\prime}}{\delta \delta^{\prime}} \right\rceil
\end{equation}
different BFSs, where $\gamma^{\prime}$ and~$\delta^{\prime}$ are the maximum and minimum absolute values, respectively, of negative reduced costs in all BFSs.
Assuming nondegenerate LPs, upper bound~\eqref{eq:bound-by-Kitahara-3} for the number of different BFSs can be regarded as an upper bound for the number of simplex iterations.\\

Inspired by these results, Tano et al. \cite{tano2019} showed bounds for the number of iterations for the simplex method with the steepest edge rule, with an additional assumption that the problem is nondegenerate.
They obtained the bounds
\[
\left\lceil m^{\frac{3}{2}} \frac{\gamma^2}{\delta^2} \log\left( \frac{\getheight c^{\top}x^0-z^*}{\getheight c^{\top}\bar{x}-z^*} \right) \right\rceil~\mathrm{and}~
(n-m) \left\lceil m^{\frac{3}{2}} \frac{\gamma^2}{\delta^2} \log\left(\frac{m\gamma}{\delta}\right) \right\rceil.\]  
\section{Analysis of the largest distance rule}
\subsection{Technical results}
We begin this section by a lemma, which  gives a lower bound for the optimal value of an LP.
\begin{lemma}[Kitahara and Mizuno~\cite{km2013}] \label{lem:Lower-Bound-Of-Optimal-Value}
Let $x^t$ be the $t$-th solution generated by the simplex method with any pivoting rule, and let $B^t$ and~$N^t$ be the basis and nonbasis, respectively, corresponding to~$x^t$.
Set $\Delta_{d,t} = -\min\{\bar{c}_j|~j\in N^t\}$.
Then, we have
\begin{equation} \label{eq:Lower-Bound-Of-Optimal-Value}
  z^* \ge c^{\top} x^t - m \gamma \Delta_{d,t}
\end{equation}
\end{lemma}

Assume that now we are at $t$-th iteration of the simplex method with basis $B^t$ and nonbasis $N^t$. 
Let $x_{j_{d,t}}$ and $x_{j_{l,t}}$ be nonbasic variables chosen by Dantzig's rule and the largest distance rule, respectively. 
We also set $\Delta_{l,t}=\bar{c}_{j_{l,t}}$.
Then from (\ref{largest}), we have
\[
\frac{\bar{c}_{j_{l,t}}}{\|a_{j_{l,t}\|}}\le \frac{\bar{c}_{j_{d,t}}}{\|a_{j_{d,t}\|}}
\]
from rearranging this inequality and the definition of $\beta$ (\ref{beta}), we obtain
\begin{equation}
\Delta_{l,t}=\bar{c}_{j_{l,t}}\le \frac{\|a_{j_{l,t}}\|}{\|a_{j_{d,t}}\|}\bar{c}_{j_{d,t}}\le \beta \Delta_{d,t}.
\label{estimate}
\end{equation} 

Next, we show that in each iteration of the simplex method with the largest distance rule, the difference between the objective value and the optimal value decreases at a constant ratio or more.

\begin{lemma} \label{lm:gap_rate}
Let $x^t$ and $x^{t+1}$ be the $t$-th and $(t+1)$-th solutions, respectively, of the simplex method with largest distance rule.
If $x^t\not=x^{t+1}$, we have the following inequality:
\begin{equation} \label{eq:gap_rate}
  c^{\top} x^{t+1} - z^* \le \left( 1 - \frac{\beta \delta}{m \gamma} \right) \left( c^{\top} x^{t} - z^* \right).
\end{equation}
\end{lemma}
\begin{proof}
The objective value decreases by $\Delta_{l,t} x_{j_{l}}^{t+1}$ at $t$ th iteration.
Note that since $x^t\not=x^{t+1}$, $x_{j_{l}}^{t+1}\not=0$ and from the definition of $\delta$, we have $x_{j_{l}}^{t+1}\ge\delta$.
Therefore, we obtain 
\begin{equation} \notag
  c^{\top} x^{t} - c^{\top} x^{t+1} = \Delta_{l,t} x_{j_{l}}^{t+1} \ge \Delta_{l,t} \delta.
\end{equation}
From inequalities~(\ref{eq:Lower-Bound-Of-Optimal-Value}) and~(\ref{estimate}), we have
\begin{equation} \notag
  \Delta_{l,t} \delta \ge \beta\Delta_{d,t} \delta \ge \beta \delta \cdot \frac{c^{\top} x^t - z^*}{m\gamma},
\end{equation}
and hence we obtain
\begin{equation} \notag
  c^{\top} x^{t} - c^{\top} x^{t+1} \ge \frac{\beta\delta}{m\gamma} \left(c^{\top} x^t - z^* \right),
\end{equation}
which leads to inequality~\eqref{eq:gap_rate}. 
\end{proof}
\subsection{A bound using the objective function}
For a positive integer $t$, let 
\[\tilde{t}=|\{t'\in\{0,1\dots,t-1\}|~x^{t'+1}\not=x^{t'}\}|\]. 
That is, $\tilde{t}$ shows how many times the solution changed from $0$-th iteration to $(t-1)$-th iteration. 
Then from Lemma \ref{lm:gap_rate}, we have 
%Applying inequality~\eqref{eq:gap_rate} from $t=0,1,2,\ldots$ in order, we have
\begin{equation} \label{eq:Gap-Rate-Initial}
  c^{\top} x^{t} - z^* \le \left( 1 - \frac{\beta\delta}{m\gamma} \right)^{\tilde{t}} \left( c^{\top} x^0 - z^* \right).
\end{equation}
Let $\bar{x}$ be a second optimal BFS, that is, a BFS whose objective value is the second smallest of all BFSs. 
Obviously, if the $t$-th solution $x^{t}$ satisfies the following inequality, $x^{t}$ is optimal:
\[
  c^{\top} x^t - z^* < c^{\top} \bar{x} - z^*.
\]
Therefore, the simplex method with the largest distance rule finds an optimal solution and terminates 
after generating at most $T$~different BFSs starting from an initial BFS~$x^0$, where $T$ is the smallest integer~$\tilde{t}$ such that the right-hand side of inequality~\eqref{eq:Gap-Rate-Initial} is less than $c^{\top} \bar{x} - z^*$. 
Using this observation and Lemma \ref{lm:gap_rate}, we have the following theorem.
\begin{theorem} \label{thm:Upper-Bound-With-Second-Optimal-Ratio}
Let $\bar{x}$ be a BFS of problem~\eqref{primal} with the second smallest objective value.
For problem~\eqref{primal}, the simplex method with the largest distance rule generates at most  
\begin{equation} \label{eq:Upper-Bound-With-Second-Optimal-Ratio}
  \left\lceil \frac{m\gamma}{\beta \delta} \log\left( \frac{\getheight c^{\top}x^0-z^*}{\getheight c^{\top}\bar{x}-z^*} \right) \right\rceil
\end{equation}
different BFSs.
\end{theorem}
\begin{proof}
As mentioned earlier, the smallest integer~$\tilde{t}$ satisfying the following inequality is an upper bound for the number of different BFSs:
\begin{equation} \notag
  \left( 1 - \frac{\beta \delta}{m\gamma} \right)^{\tilde{t}} \left( c^{\top} x^0 - z^* \right) < c^{\top} \bar{x} - z^*.
\end{equation}
Solving this inequality for $t$, we have
\begin{equation}
  \tilde{t} >  \displaystyle \left. -\log\left( \frac{c^{\top}x^0-z^*}  {c^{\top}\bar{x}-z^*}\right) \right/  \log\left(1-\frac{\beta \delta}{m\gamma}\right) .
\end{equation}
Since $ 1/x > -1/\log(1-x) $ holds in $0 < x < 1$, we obtain
\begin{equation} \notag
  \frac{m\gamma}{\beta \delta} > -\frac{1}{\displaystyle\log\left(1-\frac{\beta \delta}{m\gamma}\right)},
\end{equation}
which leads to the theorem. 
\end{proof}
\subsection{A bound independent of the objective function}
The following lemma is needed to obtain a bound independent of the objective function, for the number of BFSs generated by the simplex method with the largest distance rule.
\begin{lemma}[Kitahara and Mizuno~\cite{km2013}] \label{lem:Non-Optimal-Variable-Upper-Bound}
Let $x^{t}$ be the $t$-th solution of the simplex method and $B^t$ be the basis corresponding to~$x^{t}$.
If $x^t$ is not optimal, there exists $\bar{\jmath}\in B^t$ that satisfies the following conditions:
\begin{equation} \label{eq:Optimal-Dual-Slack-Lower-Bound}
  x_{\bar{\jmath}}^t > 0 \qquad \text{and} \qquad s_{\bar{\jmath}}^* \ge \frac{1}{m x_{\bar{\jmath}}^t} \left( c^{\top} x^t - z^* \right),
\end{equation}
where $z^*$ is the slack vector corresponding to an optimal basic solution of the dual problem~(\ref{dual}).
Furthermore, the $k$-th solution $x^{k}$ satisfies
\begin{equation} \label{eq:Non-Optimal-Variable-Upper-Bound}
  x_{\bar{\jmath}}^{k} \le m x_{\bar{\jmath}}^t \, \frac{c^{\top} x^{k} - z^*}{c^{\top} x^t - z^*}
\end{equation}
for an arbitrary positive integer $k$.
\end{lemma}
Now we are ready to prove a bound independent of the objective function, for the number of BFSs generated by the simplex method with largest distance rule.
\begin{theorem} \label{thm:Upper-Bound-Without-Objective-Function-Ratio}
When applying the simplex method with the largest distance rule to problem~\eqref{primal}, it generates at most
\begin{equation} \label{eq:Upper-Bound-Without-Objective-Function-Ratio}
  (n-m) \left\lceil \frac{m\gamma}{\beta \delta} \log\left(\frac{m\gamma}{\delta}\right) \right\rceil.
\end{equation}
different BFSs.
\end{theorem}
\begin{proof}
Let $r$ be a positive integer. 
Moreover, let $x^t$ and $x^{t+r}$ be the $t$- and $(t+r)$-th solutions, respectively, of the simplex method. 
Assume that excluding $x^t$, $\tilde{r}$ different BFSs are generated between $t$-th and ($t+r$)-th iteration.
In addition, let $B^t$ be the basis corresponding to~$x^t$. 
Then, by Lemmas~\ref{lm:gap_rate} and~\ref{lem:Non-Optimal-Variable-Upper-Bound} and the definition of $\gamma$, there exists $\bar{\jmath}\in B^t$ such that 
\begin{equation} \notag
  x_{\bar{\jmath}}^{t+r} \le m x_{\bar{\jmath}}^t \left( 1 - \frac{\beta\delta}{m\gamma} \right)^{\tilde{r}} \le m \gamma \left( 1 - \frac{\beta\delta}{m\gamma} \right)^{\tilde{r}}.
\end{equation}
When $\tilde{r} \geq (m\gamma)/(\beta\delta) \cdot \log\left(m\gamma/\delta\right)$, the rightmost term above is less than~$\delta$, and $x_{\bar{\jmath}}^{t+r}$ must be fixed to zero by the definition of~$\delta$.\\

If an optimal solution is not obtained after $r$~iterations satisfying the above inequality, we can apply the same procedure again.
%The number of variables that can be chosen as a basis decreases by one through each process.
Let us fix a basic optimal solution $(y^*,s^*)$ of the dual problem (\ref{dual}) and let $B^*$ and $N^*$ be basis and nonbasis for the solution. 
Since the event described in Lemma \ref{lem:Non-Optimal-Variable-Upper-Bound} can occur for $\bar{j} \in N^*$ at most once and $|N^*|=n-m$, 
the value fixing process for primal variables occurs at most $n-m$ times. 
Thus we have the desired result.
\end{proof}
\begin{corollary}
When the problem is nondegenerate, the simplex method with the largest distance rule finds an optimal solution in at most
\[
\min\left\{
  \left\lceil \frac{m\gamma}{\beta \delta} \log\left( \frac{\getheight c^{\top}x^0-z^*}{\getheight c^{\top}\bar{x}-z^*} \right) \right\rceil,~
(n-m) \left\lceil \frac{m\gamma}{\beta \delta} \log\left(\frac{m\gamma}{\delta}\right) \right\rceil.
\right\}
\]
iterations.
\end{corollary}
\section{Conclusion}
In this paper, we analyzed the simplex method with the largest distance rule and obtained two bounds for the number of different BFSs. 
This pivoting rule has not been extensively studied from a theoretical perspective.
In particular, our results show that the complexity of the largest distance rule can be characterized within the existing analytical framework, with an explicit dependence on the geometric parameter $\beta$.
We hope the current study sheds light on this new pivoting rule and encourages more research. \\

\noindent Data availability: No data was used in this study.

\end{document}